\documentclass[11pt]{amsart}
\usepackage[all,dvips]{xy}

\usepackage{slashbox}

\usepackage{amsfonts,epsfig}
\usepackage{latexsym}
\usepackage{amssymb}
\usepackage{amsmath}
\usepackage{amsthm}
\usepackage{graphics}
\usepackage[all]{xy}
\usepackage[T2A]{fontenc}
\usepackage{multirow}
\usepackage{array, booktabs, ctable}
\usepackage{enumerate}
\usepackage{booktabs}
\usepackage{longtable}
\usepackage{multirow}
\usepackage[pagebackref]{hyperref}
\usepackage{xcolor,colortbl}

\usepackage[figurename=Diagram]{caption}

\renewcommand*{\arraystretch}{1.3}

\newtheorem{theorem}{Theorem}
\newtheorem{corollary}[theorem]{Corollary}

\newtheorem{lemma}[theorem]{Lemma}
\newtheorem{proposition}[theorem]{Proposition}

\newtheorem*{ack}{Acknowledgements}
\newtheorem*{example}{Example}

\newtheorem*{remark}{Remark}

\newcommand{\cC}{{\mathcal{C}}}
\DeclareMathOperator{\cm}{CM}

\DeclareMathOperator{\tors}{tors}

\newcommand{\Q}{\mathbb Q}
\newcommand{\Qbar}{{\overline{\mathbb Q}}} 
\newcommand{\Z}{\mathbb Z}

\newcommand{\Gal}{\operatorname{Gal}}
\newcommand{\Aut}{\operatorname{Aut}}

\newcommand{\GL}{\operatorname{GL}}

\begin{document}

\bibliographystyle{plain}
\title[Torsion of rational elliptic curves over quintic fields]{Complete classification of the torsion structures of rational elliptic curves over quintic number fields}

\author{Enrique Gonz\'alez--Jim\'enez}
\address{Universidad Aut{\'o}noma de Madrid, Departamento de Matem{\'a}ticas, Madrid, Spain}
\email{enrique.gonzalez.jimenez@uam.es}
\subjclass[2010]{Primary: 11G05, 14G05; Secondary: 14H52, 11R21}
\keywords{Elliptic curves, torsion subgroup, rationals, quintic number fields.}
\thanks{The author was partially  supported by the grant MTM2015--68524--P.}

\begin{abstract}
We classify the possible torsion structures of rational elliptic curves over quintic number fields. In addition, let $E$ be an elliptic curve defined over $\Q$ and let $G = E(\Q)_{\tors}$ be the associated torsion subgroup. We study, for a given $G$, which possible groups $G \subseteq H$ could appear such that $H=E(K)_{\tors}$, for $[K:\Q]=5$. In particular, we prove that at most there is one quintic number field $K$ such that the torsion grows in the extension $K/\Q$, i.e.,  $E(\Q)_{\tors} \subsetneq E(K)_{\tors}$.
\end{abstract}

\maketitle

\section{Introduction}
Let $E/K$ be an elliptic curve defined over a number field $K$. The Mordell-Weil Theorem states that the set of $K$-rational points, $E(K)$, is a finitely generated abelian group. Denote by $E(K)_{\tors}$, the torsion subgroup of $E(K)$, which is isomorphic to $\cC_m\times\cC_n$ for two positive integers $m,n$, where $m$ divides $n$ and where $\cC_n$ is a cyclic group of order $n$. 


One of the main goals in the theory of elliptic curves is to characterize the possible torsion structures over a given number field, or over all number fields of a given degree.  In 1978 Mazur \cite{Mazur1978} published a proof of Ogg's conjecture (previously established by Beppo Levi), a milestone in the theory of elliptic curves. In that paper, he proved that the possible torsion structures over $\Q$ belong to the set:
$$
\Phi(1) = \left\{ \cC_n \; | \; n=1,\dots,10,12 \right\} \cup \left\{ \cC_2 \times \cC_{2m} \; | \; m=1,\dots,4 \right\},
$$
and that any of them occurs infinitely often. A natural generalization of this theorem is as follows. Let $\Phi(d)$ be the set of possible isomorphic torsion structures $E(K)_{\tors}$, where $K$ runs through all number fields $K$ of degree $d$ and $E$ runs through all elliptic curves over $K$. Thanks to the  uniform boundedness theorem \cite{Merel}, $\Phi(d)$ is a finite set. Then the problem is to determine $\Phi(d)$. Mazur obtained the rational case ($d=1$). The generalization to quadratic fields ($d=2$) was obtained by Kamienny, Kenku and Momose \cite{K92, KM88}. For $d\ge 3$ a complete answer for this problem is still open, although there have been some advances in the last years.

However, more is known about  the subset $\Phi^\infty(d)\subseteq \Phi(d)$ of torsion subgroups that arise for infinitely many $\Qbar$-isomorphism classes of elliptic curves defined over number fields of degree $d$. For $d=1$ and $d=2$ we have $\Phi^\infty(d)=\Phi(d)$, the cases $d=3$ and $d=4$ have been determined by Jeon et al. \cite{JKP04,JKP06}, and recently the cases $d=5$ and $d=6$ by Derickx and Sutherland \cite{DS16}.

Restricting our attention to the complex multiplication case, we denote  $\Phi^{\cm}(d)$ the analogue of the set $\Phi(d)$ but restricting to elliptic curves with complex multiplication (CM elliptic curves in the sequel). In 1974 Olson \cite{Olson74} determined the set of possible torsion structures over $\Q$ of CM elliptic curves:
$$
\Phi^{\cm}(1)=\left\{ \cC_1\,,\,  \cC_2\,,\,  \cC_3\,,\,  \cC_4\,,\,  \cC_6\,,\, \cC_2\times\cC_2\right\}.
$$
The quadratic and cubic cases were determined by Zimmer et al. \cite{MSZ89,FSWZ90,PWZ97}; and recently, Clark et al. \cite{Clark2014} have computed the sets $\Phi^{\cm}(d)$, for $4\le d\le 13$. In particular, they proved 
$$
\Phi^{\cm}(5)=\Phi^{\cm}(1)\cup\{\,\cC_{11}\,\}.
$$

In addition to determining $\Phi(d)$, there are many authors interested in the question of how the torsion grows when the field of definition is enlarged. We focus our attention when the underlying field is $\Q$. In analogy to $\Phi(d)$, let $\Phi_\Q(d)$  be the subset of $\Phi(d)$  such that $H\in \Phi_\Q(d)$ if there is an elliptic curve $E/\Q$ and a number field $K$ of degree $d$ such that $E(K)_\text{tors}\simeq H$. One of the first general result is due to Najman \cite{N15a}, who determined $\Phi_{\mathbb Q}(d)$  for $d = 2, 3$. Chou \cite{chou} has given a partial answer to the classification of $\Phi_{\mathbb Q}(4)$. Recently, the author with Najman \cite{GJN16} have completed the classification of $\Phi_{\mathbb Q}(4)$ and $\Phi_\Q(p)$ for $p$ prime. Moreover, in \cite{GJN16} it has been proved that $E(K)_{\tors}=E(\Q)_{\tors}$ for all elliptic curves $E$ defined over $\Q$ and all number fields $K$ of degree $d$, where $d$ is not divisible by a prime $\leq 7$. In particular, $\Phi_\Q(d)=\Phi(1)$ if $d$ is not divisible by a prime $\leq 7$.

\

Our first result determines $\Phi_\Q(5)$.

\begin{theorem}\label{main1} The sets $\Phi_\Q(5)$ and $\Phi^{\cm}_\Q(5)$ are given by
$$
\begin{array}{lcl}
\Phi_{\mathbb Q}(5) & = & \left\{ \, \cC_n \; | \; n=1,\dots,12,25 \right\} \cup \left\{\, \cC_2 \times \cC_{2m} \; | \; m=1,\dots,4 \right\},\\
\Phi^{\cm}_{\mathbb Q}(5)& = & \left\{\, \cC_1\,,\, \cC_2\,,\,\cC_3\,,\,\cC_4\,,\,\cC_6\,,\,\cC_{11}\,,\, \cC_2 \times \cC_{2}\,\right\}.
\end{array}
$$
\end{theorem}

\begin{remark}
$\Phi_{\mathbb Q}(5)=\Phi_{\mathbb Q}(1)\cup\{\,\cC_{11},\,\cC_{25}\,\}$ and $\Phi^{\cm}_{\mathbb Q}(5)=\Phi^{\cm}(5)=\Phi^{\cm}(1)\cup\{\,\cC_{11}\,\}$.
\end{remark}

For a fixed $G \in \Phi(1)$, let $\Phi_\Q(d,G)$ be the subset of $\Phi_\Q(d)$ such that $E$ runs through all elliptic curves over $\Q$  with $E(\Q)_{\tors}\simeq G$. For each $G\in\Phi(1)$ the sets $\Phi_{\mathbb Q}(d,G)$ have been determined for $d=2$ in \cite{K97, GJT14}, for $d=3$ in \cite{GJNT15} and partially for $d=4$ in \cite{GL16}.

\

Our second result determines $\Phi_\Q(5)$ for any $G \in \Phi(1)$.

\begin{theorem}\label{main2}
For $G \in \Phi(1)$, we have $\Phi_\Q(5,G)= \{G\}$, except in the following cases:
$$
\begin{array}{|c|c|}
\hline
G & \Phi_\mathbb{Q} \left(5,G \right)\\
\hline
\cC_1 & \left\{ \cC_1\,,\,\cC_{5}\,,\,   \cC_{11} \, \right\} \\
\hline
\cC_2 &\begin{array}{c} \left\{ \cC_2\,,\cC_{10}\, \right\} 
   \end{array}\\ 
\hline
\cC_5 & \left\{ \cC_5\,,\, {  \cC_{25} }\,\right\} \\
\hline
\end{array}
$$
Moreover, there are infinitely many $\Qbar$-isomorphism classes of elliptic curves $E/\Q$ with $H\in \Phi_\mathbb{Q} \left(5,G \right)$, except for the case $H=\cC_{11}$ where only the elliptic curves \texttt{121a2}, \texttt{121c2}, \texttt{121b1} have eleven torsion over a quintic number field. 
\end{theorem}
In fact, it is possible to give a more detailed description of how the torsion grows. For this purpose for any $G\in \Phi(1)$ and any positive integer $d$,  we define the set
$$
\mathcal{H}_{\Q}(d,G) = \{ S_1,...,S_n \}
$$
where $S_i= \left[ H_1,...,H_m \right]$ is a list of groups $H_j \in \Phi_{\mathbb Q}(d,G) \setminus \{ G \}$, such that, for each $i=1,\ldots,n$, there exists an elliptic curve $E_i/\Q$ that satisfies the following properties:
\begin{itemize}
\item $E_i({\mathbb Q})_{\text{tors}} \simeq G$, and
\item there are number fields $K_1,...,K_m$ (non--isomorphic pairwise) whose degrees divide $d$ with $E_i ( K_j )_\text{tors} \simeq H_j$, for all $j=1,...,m$;  and for each $j$ there does not exist $K'_j\subset K_j$ such that $E_i ( K'_j )_\text{tors} \simeq H_j$.
\end{itemize}
We are allowing the possibility of two (or more) of the $H_j$ being isomorphic. The above sets have been completely determined for the quadratic case ($d=2$)  in \cite{GJT15}, for the cubic case ($d=3$) in  \cite{GJNT15} and computationally conjectured for the quartic case ($d=4$) in \cite{GL16}. The quintic case ($d=5$) is treated in this paper, and the next result determined $\mathcal{H}_{\Q}(5,G)$ for any $G\in\Phi(1)$:


\begin{theorem}\label{main3}
For $G \in \Phi(1)$, we have $\mathcal{H}_{\Q}(5,G)=\emptyset$, except in the following cases:
\begin{center}
\begin{tabular}{|c|c|}
\hline
$G$ & $\mathcal{H}_{\Q}(5,G)$\\
\hline
\multirow{2}{*}{$\cC_1$} & $\cC_5$ \\
\cline{2-2}
& {$\cC_{11}$}   \\
\hline
{$\cC_2$}  & {$\cC_{10}$}   \\
\hline
{$\cC_5$}  & {$\cC_{25}$}   \\
\hline
\end{tabular}
\end{center}
In particular, for any elliptic curve $E/\Q$, there is at most one quintic number field $K$, up to isomorphism, such that $E(K)_{\tors}\ne E(\Q)_{\tors}$.
\end{theorem}

\begin{remark}
Notice that for any CM elliptic curve $E/\Q$ and any quintic number field $K$ it has $E(K)_{\tors}=E(\Q)_{\tors}$, except to the elliptic curve \texttt{121b1} and $K=\Q(\zeta_{11})^+=\Q(\zeta_{11}+\zeta_{11}^{-1})$ where $E(\Q)_{\tors}\simeq \cC_1$ and  $E(K)_{\tors}\simeq \cC_{11}$.
\end{remark}

Let us define
$$
h_{\Q}(d)=\max_{G \in \Phi(1)} \Big\{ \#S \; \Big| \; S \in \mathcal{H}_{\Q}(d,G) \Big\}.
$$
The values $h_{\Q}(d)$ have been computed for $d=2$ and $d=3$ in \cite{GJT15} and \cite{GJNT15} respectively. For $d=4$ we computed a lower bound in \cite{GL16}. For $d=5$ we have:
\begin{corollary}
$h_\Q(5)=1$.
\end{corollary}

\begin{remark}
In particular, we have deduced the following:
$$
\begin{array}{|c|c|c|c|c|}
\hline
d & 2 & 3 & 4 & 5\\
\hline
h_{\Q}(d) & 4 & 3 & \ge 9& 1\\
\hline  
\end{array}
$$
\end{remark}

\subsubsection*{Notation} 
 We will use the Antwerp--Cremona tables and labels \cite{antwerp,cremonaweb} when referring to specific elliptic curves over $\Q$.

For conjugacy classes of subgroups of $\GL_2(\Z/p\Z)$ we will use the labels introduced by Sutherland in \cite[\S 6.4]{Sutherland2}.

We will write $G\simeq H$ (or $G\lesssim H$) for the fact that $G$ is isomorphic to $H$ (or to a subgroup of $H$ resp.) without further detail on the precise isomorphism.

For a positive integer $n$  we will write $\varphi(n)$ for the Euler-totient function of $n$.

We use $\mathcal O$ to denote the point at infinity of an elliptic curve (given in Weierstrass form).


\section{Mod $n$ Galois representations associated to elliptic curves}
Let $E/\Q$ be an elliptic curve and $n$ a positive integer. We denote by $E[n]$ the $n$-torsion subgroup of $E(\Qbar)$, where $\Qbar$ is a fixed algebraic closure of $\Q$. That is, $E[n]=\{P\in E(\Qbar)\,|\, [n]P=\mathcal O\}$. The absolute Galois group $\Gal(\Qbar/\Q)$ acts on $E[n]$ by its action on the coordinates of the points, inducing a Galois representation  
$$
\rho_{E,n}\,:\,\Gal(\Qbar/\Q)\longrightarrow \Aut(E[n]). 
$$
Notice that since $E[n]$ is a free $\Z/n\Z$-module of rank $2$, fixing a basis $\{P,Q\}$ of $E[n]$, we identify $ \Aut(E[n])$ with $\GL_2(\Z/n\Z)$. Then we rewrite the above Galois representation as
$$
\rho_{E,n}\,:\,\Gal(\Qbar/\Q)\longrightarrow\GL_2(\Z/n\Z).
$$
Therefore we can view $\rho_{E,n}(\Gal(\Qbar/\Q))$ as a subgroup of $\GL_2(\Z/n\Z)$, determined uniquely up to conjugacy, and denoted by $G_E(n)$ in the sequel. Moreover, $\Q(E[n])=\{x,y\,|\, (x,y)\in E[n]\}$ is Galois and since $\ker{\rho_{E,n}}=\Gal(\Qbar/\Q(E[n]))$, we deduce that $G_E(n)\simeq\Gal(\Q(E[n])/\Q)$.

Let $R=(x(R),y(R))\in E[n]$ and $\Q(R)=\Q(x(R),y(R))\subseteq \Q(E[n])$, then by Galois theory there exists a subgroup $\mathcal H_R$ of  $\Gal(\Q(E[n])/\Q)$ such that $\Q(R)=\Q(E[n])^{\mathcal H_R}$. In particular, if we denote by $H_R$ the image of $\mathcal H_R$ in $\GL_2(\Z/n\Z)$, we have:
\begin{itemize}
\item $[\Q(R):\Q]=[G_E(n):H_R]$.
\item $\Gal(\widehat{\Q(R)}/\Q)\simeq G_E(n)/N_{G_E(n)}(H_R)$, where $\widehat{\Q(R)}$ denotes the Galois closure of $\Q(R)$ in $\Qbar$, and $N_{G_E(n)}(H_R)$ denotes the normal core of $H_R$ in $G_E(n)$.
\end{itemize}
We have deduced the following result.
\begin{lemma}\label{lem_order_n}
Let $E/\Q$ be an elliptic curve, $n$ a positive integer and $R\in E[n]$. Then
$[\Q(R):\Q]$ divides $|G_E(n)|$. In particular $[\Q(R):\Q]$ divides $|\GL_2(\Z/n\Z)|$. \end{lemma}

In practice, given the conjugacy class of $G_E(n)$ we can deduce the relevant arithmetic-algebraic properties of the fields of definition of the $n$-torsion points: since $E[n]$ is a free $\Z/n\Z$-module of rank $2$, we can identify the $n$-torsion points with $(a,b)\in (\Z/n\Z)^2$ (i.e. if $R\in E[n]$ and $\{P,Q\}$ is a $\Z/n\Z$-basis of  $E[n]$, then there exist $a,b\in \Z/n\Z$ such that $R=aP+bQ$). Therefore $H_R$ is the stabilizer of $(a,b)$ by the action of $G_E(n)$ on $(\Z/n\Z)^2$. In order to compute all the possible degrees (jointly with the Galois group of its Galois closure in $\Qbar$) of the fields of definition of the $n$-torsion points we run over all the elements of $(\Z/n\Z)^2$ of order $n$.

Now, observe that $\langle R\rangle\subset E[n]$ is a subgroup of order $n$. Equivalently, $E/\Q$ admits a cyclic $n$-isogeny (non-rational in general). The field of definition of this isogeny is denoted by $\Q(\langle R\rangle)$. A similar argument could be used to obtain a description of $\Q(\langle R\rangle)$ using Galois theory. In particular, if  $\langle R\rangle$ is $\Gal(\Qbar/\Q)$-stable then the isogeny is defined over $\Q$. To compute the relevant arithmetic-algebraic properties of the field $\Q(\langle R\rangle)$ is similar to the case $\Q(R)$, replacing the pair $(a,b)$ by the $\Z/n\Z$-module of rank $1$ generated by $(a,b)$ in $(\Z/n\Z)^2$.

In the case $E/\Q$ be a non-CM elliptic curve and $p\le 11$ be a prime, Zywina \cite{zywina} has described all the possible subgroups of $\GL_2(\Z/p\Z)$ that occur as $G_E(p)$. 

%

For each possible subgroup $G_E(p)\subseteq \GL_2(\Z/p\Z)$ for $p\in\{2, 3, 5, 11\}$, Table \ref{tableSutherland} \href{http://www.uam.es/personal_pdi/ciencias/engonz/research/tables/tors5/Table1.txt}{\color{blue}lists} in the first and second column the corresponding labels in Sutherland and Zywina notations, and the following data: 
\begin{itemize}
\item[$d_0$:] the index of the largest subgroup of $G_E(p)$ that fixes a $\Z/p\Z$-submodule of rank $1$ of $E[p]$; equivalently, the degree of the minimal extension $L/\Q$ over which $E$ admits a $L$-rational $p$-isogeny.
\item[$d_v$:] is the index of the stabilizers of $v\in(\Z/p\Z)^2$, $v\ne (0,0)$, by the action of $G_E(p)$ on $(\Z/p\Z)^2$; equivalently, the degrees of the extension $L/\Q$ over which $E$ has a $L$-rational point of order $p$.
\item[$d$:] is the order of $G_E(p)$; equivalently, the degree of the minimal extension $L/\Q$ for which $E[p]\subseteq E(L)$.
\end{itemize}
Note that Table \ref{tableSutherland} is partially extracted from Table 3 of \cite{Sutherland2}. The difference is that \cite[Table 3]{Sutherland2} only lists the minimum of $d_v$, which is denoted by $d_1$ therein.

\begin{table}
 \begin{tabular}{cc}
\begin{tabular}{@{\hskip -6pt}cllccc}
& \multicolumn{1}{c}{Sutherland}  &  \multicolumn{1}{c}{Zywina}  &$d_0$&$d_v $ & $d$\\\toprule
&\texttt{2Cs} & $G_1$  & 1 & 1 & 1\\ 
&\texttt{2B} &$G_2$ &1 & 1\,,\,2 & 2\\ 
&\texttt{2Cn} & $G_3$ & 3 & 3  & 3\\
 \multicolumn{3}{c}{$\,\,\,\,\,\,\GL(2,\Z/2\Z)$} & 3& 3& 6\\
\midrule
&\texttt{3Cs.1.1} &$H_{1,1}$ & 1& 1\,,\,2 & 2\\
&\texttt{3Cs} & $G_1$ &  1 & 2\,,\,4  & 4 \\
 &\texttt{3B.1.1} & $H_{3,1}$& 1 & 1\,,\,6  &6 \\
&\texttt{3B.1.2} &  $H_{3,2}$& 1 & 2\,,\,3  & 6\\
&\texttt{3Ns} &$G_2$ & 2 & 4  & 8\\
&\texttt{3B} & $G_3$ & 1 & 2\,,\,6  & 12\\
&\texttt{3Nn} & $G_4$ & 4 & 8  & 16\\
 \multicolumn{3}{c}{$\,\,\,\,\,\,\GL(2,\Z/3\Z)$} & 4& 8& 48\\
\midrule
&\texttt{11B.1.4} & $H_{1,1}$ & 1 & 5\,,\,110 & 110\\
&\texttt{11B.1.5} & $H_{2,1}$ & 1 & 5\,,\,110  & 110\\
&\texttt{11B.1.6} & $H_{2,2}$ & 1 & 10\,,\,55  & 110\\
&\texttt{11B.1.7} & $H_{1,2}$ & 1 & 10\,,\,55  & 110\\
&\texttt{11B.10.4} & $G_{1}$ & 1 & 10\,,\,110  & 220\\
&\texttt{11B.10.5} & $G_{2}$ & 1 & 10\,,\,110  & 220\\
&\texttt{11Nn} & $G_{3}$ & 12 & 120 & 240\\
 \multicolumn{3}{c}{$\,\,\,\,\,\,\GL(2,\Z/11\Z)$} & 12& 120& 13200\\
\bottomrule
\end{tabular}
&\qquad\quad
\begin{tabular}{@{\hskip -6pt}cllccc}
& \multicolumn{1}{c}{Sutherland}  &  \multicolumn{1}{c}{Zywina}   &$d_0$&$d_v $ & $d$\\\toprule
&\texttt{5Cs.1.1} & $H_{1,1}$ & 1 & 1\,,\,4  & 4\\
&\texttt{5Cs.1.3} & $H_{1,2}$& 1 & 2\,,\,4  & 4\\
&\texttt{5Cs.4.1} & $G_1$ & 1 & 2\,,\,4\,,\,8  & 8\\
&\texttt{5Ns.2.1} & $G_3$ & 2 & 8\,,\,16  & 16\\
&\texttt{5Cs} &  $G_2$ & 1 & 4  & 16\\
&\texttt{5B.1.1} & $H_{6,1}$ & 1 & 1\,,\,20  & 20\\
&\texttt{5B.1.2} & $H_{5,1}$ & 1 & 4\,,\,5  & 20\\
&\texttt{5B.1.4} & $H_{6,2}$ & 1 & 2\,,\,20  & 20\\
&\texttt{5B.1.3} & $H_{5,2}$ & 1 & 4\,,\,10 & 20\\
&\texttt{5Ns} & $G_{4}$ & 2 & 8\,,\,16  & 32\\
&\texttt{5B.4.1} & $G_{6}$ & 1 & 2\,,\,20  & 40\\
&\texttt{5B.4.2} & $G_{5}$ & 1 & 4\,,\,10  & 40\\
&\texttt{5Nn} & $G_{7}$ & 6 & 24  & 48\\
&\texttt{5B} & $G_{8}$ & 1 & 4\,,\,20 & 80\\
&\texttt{5S4} & $G_{9}$ & 6 & 24  & 96 \\
\multicolumn{3}{c}{$\,\,\,\,\,\,\GL(2,\Z/5\Z)$}  & 6& 24& 480\\
 \bottomrule
  & & & &  \\
  & & & &  \\
  & & & &  \\
 & & & &  \\
\end{tabular}
\bigskip
\smallskip
\end{tabular}
\caption{Image groups $G_E(p)$, for $p\in\{2,3,5,11\}$, for non-CM elliptic curves $E/\Q$.}\label{tableSutherland}
\vspace{10pt}
\end{table}
\smallskip


For the CM case, Zywina  \cite[\S 1.9]{zywina} gives a complete description of $G_E(p)$ for any prime $p$.

\section{Isogenies.}
In this paper a rational $n$-isogeny of an elliptic curve $E/\Q$ is a (surjective) morphism $E\longrightarrow E'$ defined over $\Q$ where $E'/\Q$ and the kernel is cyclic of order $n$. The rational $n$-isogenies of elliptic curves over $\Q$, have been described completely in the literature, for all $n\geq 1$. The following result gives all the possible values of $n$.

\begin{theorem}[\cite{Mazur1978, kenku39,kenku65,kenku169,kenku125}]\label{isog}
Let $E/\Q$ be an elliptic curve with a rational $n$-isogeny. Then $n \le 19$ or $n\in\{21, 25, 27, 37, 43, 67, 163\}$.
\end{theorem}

A direct consequence of the Galois theory applied to the theory of cyclic isogenies is the following (cf. Lemma 3.10 \cite{chou}). 

\begin{lemma}\label{lem_chou}
Let $E/\Q$ be an elliptic curve such that $E(K)[n]\simeq \cC_n$ over a Galois extension $K/\Q$. Then $E$ has a rational $n$-isogeny.
\end{lemma}

%

\section{$\mathcal P$-primary torsion subgroup} 

Let $E/K$ be an elliptic curve defined over a number field $K$. For a given set of primes $\mathcal P \subset \mathbb Z$, let $E(K)[{\mathcal P}^\infty]$ denote the $\mathcal P$-primary torsion subgroup of $E(K)_{\tors}$, that is, the direct product of the $p$-Sylow subgroups of $E(K)$ for $p\in\mathcal P$. If $\mathcal P =\{p\}$, let us denote by $E(K)[{p}^\infty]$.

\begin{proposition}\label{prop11}
Let $E/\Q$ be an elliptic curve and $K/\Q$ be a quintic number field. 
\begin{enumerate}
\item\label{i1_prop10} If $P$ is a point of prime order $p$ in $E(K)$, then $p\in\{2,3,5,7,11\}$.
\item\label{i2_prop10} If $E(K)[n]=E[n]$, then $n= 2$.
\end{enumerate}
\end{proposition}

\begin{proof}

(\ref{i1_prop10}) Lozano-Robledo \cite{Lozano} has determined that the set of primes $p$ for which there exists a number field $K$ of degree $\le 5$ and an elliptic curve $E/\Q$ such that the $p$ divides the order of $E(K)_{\tors}$ is given by $S_\Q(5)=\{2,3,5,7,11,13\}$. Then to finish the proof we must remove the prime $p=13$. This follows from Lemma \ref{lem_order_n} since $5$ does not divide the order of $\GL_2(\mathbb F_{13})$, that is $2^5 \cdot3^2\cdot 7 \cdot13$.


(\ref{i2_prop10}) Let $E/K$ be the base change of $E$ over the number field $K$.  If $E[n]\subseteq E(K)$ then $\Q(\zeta_n)\subseteq K$. In particular $\varphi(n)\,|\, [K:\Q]$. The only possibility if $[K:\Q]=5$ is $n=2$.

\end{proof}

\subsection{$p$-primary torsion subgroup $(p\ne 5, 11)$} 
\begin{lemma}\label{lem_p}
Let $E/\Q$ be an elliptic curve and $K/\Q$ a quintic number field. Then, for any prime $p\ne 5,11$:
$$
E(K)[p^\infty]=E(\Q)[p^\infty].
$$
In particular, if $P\in E(K)[p^\infty]$ and $p^n$ is its order, then $n\le 3, 2,1,$ if $p=2,3,7$, respectively, and $n=0$ otherwise.
\end{lemma}
\begin{proof}
Let $P\in E(K)[p^n]$. By Lemma \ref{lem_order_n}, $[\Q(P):\Q]$ divides $|\GL_2(\Z/p^n\Z)|=p^{4n-3}(p^2-1)(p-1)$. If $p\in\{2,3,7\}$ then $\Q(P)=\Q$. Together with Proposition \ref{prop11} (\ref{i2_prop10}), we deduce $E(K)[p^\infty]=E(\Q)[p^\infty]$. If $p\ge 13$ and $n>0$, then $[p^{n-1}]P\in E(K)$ is a point or order $p$, a contradiction with Proposition \ref{prop11} (\ref{i1_prop10}). That is, $E(K)[p^\infty]=E(\Q)[p^\infty]=\{\mathcal O\}$ if $p\ge 13$.
\end{proof}

\subsection{$5$-primary torsion subgroup} 
\begin{lemma}\label{lem5}
Let $E/\Q$ be an elliptic curve and $K/\Q$ a quintic number field. Then 
$$
E(K)[5^\infty]\lesssim \cC_{25}.
$$
In particular if $E(K)[5^\infty]\ne \{\mathcal O\}$ then $E$ has non-CM. Moreover:
\begin{enumerate}
\item\label{i2_lem5} if $E(\Q)[5^\infty]\simeq \cC_{5}$, then $G_E(5)$ is labeled \texttt{5B.1.1} or \texttt{5Cs.1.1};
\item\label{i3_lem5} if $E(K)[5^\infty]\simeq \cC_{5}$ and $E(\Q)[5^\infty]=\{\mathcal O\}$, then $G_E(5)$ is labeled \texttt{5B.1.2};
\item\label{i1_lem5} if $E(K)[5^\infty]\simeq \cC_{25}$, then $E(\Q)[5^\infty]\simeq \cC_{5}$. Moreover, $K$ is Galois if  $G_E(5)$ is labeled \texttt{5B.1.1}.
\end{enumerate}
\end{lemma}
\begin{proof}
First suppose that $E$ has CM. Then by the classification $\Phi_\Q^{\cm}(5)$ we deduce that $E(K)[5^\infty]= \{\mathcal O\}$. From now on we assume that $E$ is non-CM. First, it is not possible $E[5]\subseteq E(K)$ by Proposition \ref{prop11} (\ref{i2_prop10}). Now, the characterization of $\Phi(1)$ tells us that $E(\Q)[5^\infty]\lesssim \cC_5$. We observe in Table \ref{tableSutherland} that $d_v=1$ (resp. $d_v=5$) for some $v\in (\Z/5\Z)^2$ of order $5$ if and only if $G_E(5)$ is labeled by \texttt{5Cs.1.1} or \texttt{5B.1.1} (resp. \texttt{5B.1.2}), which proves (\ref{i2_lem5}) (resp. (\ref{i3_lem5})). We are going to prove that $E(K)[5^\infty]\lesssim \cC_{25}$. First, we prove (\ref{i1_lem5}). Assume that there exists a quintic number field $K$ such that $E(K)[25]=\langle P \rangle \simeq \cC_{25}$. Then $G_E(25)$ satisfies: 
$$
G_E(25)\equiv\, G_E(5) \,\,(\mbox{mod $5$})\qquad\mbox{and}\qquad [G_E(25):H_P]=5.
$$
Note that in general we do not have an explicit description of $G_E(25)$, but using \texttt{Magma} \cite{magma} we do a simulation with subgroups of $GL_2(\Z/25\Z)$. 

First assume that $G_E(5)$ is labeled by \texttt{5B.1.2}, then $G_E(5)$ is conjugate in $\GL_2(\Z/5\Z)$ to the subgroup (cf. \cite[Theorem 1.4 (iii)]{zywina})
$$
H_{5,1}=\left\langle\begin{pmatrix}2 &0 \\ 0 & 1\end{pmatrix},\begin{pmatrix}1 &1 \\ 0 & 1\end{pmatrix}\right\rangle\subset \GL_2(\Z/5\Z).
$$
Since we do not have a characterization of $G_E(25)$, \href{http://www.uam.es/personal_pdi/ciencias/engonz/research/tables/tors5/Lemma10.txt}{\color{blue}{we check using \texttt{Magma}}} that for any subgroup $G$ of $GL_2(\Z/25\Z)$ satisfying $G \equiv H\,\, (\mbox{mod $5$})$ for some conjugate $H$ of $H_{5,1}$ in $\GL_2(\Z/5\Z)$, and for any $v\in (\Z/25\Z)^2$ of order $25$, we have $[G:G_v]\ne 5$ (where $G_v$ be the stabilizer of $v$ by the action of $G$ on $(\Z/25\Z)^2$). Therefore for any point $P\in E[25]$ it has $[G_E(25):H_P]\ne 5$. In particular this proves that if $G_E(5)$ is labeled  by \texttt{5B.1.2}, then there is not $5^n$-torsion over a quintic number field, for $n>1$. This finishes the first part of (\ref{i1_lem5}).

Now assume that $G_E(5)$ is labeled by \texttt{5B.1.1}. That is, $G_E(5)$ is conjugate in $\GL_2(\Z/5\Z)$ to the subgroup (cf. \cite[Theorem 1.4 (iii)]{zywina})
$$
H_{6,1}=\left\langle\begin{pmatrix}1 &0 \\ 0 & 2\end{pmatrix},\begin{pmatrix}1 &1 \\ 0 & 1\end{pmatrix}\right\rangle\subset \GL_2(\Z/5\Z).
$$
A similar argument as the one used before, \href{http://www.uam.es/personal_pdi/ciencias/engonz/research/tables/tors5/Lemma10.txt}{\color{blue}{we check}} that for any subgroup $G$ of $GL_2(\Z/25\Z)$ satisfying $G \equiv H\,\, (\mbox{mod $5$})$ for some conjugate $H$ of $H_{6,1}$ in $\GL_2(\Z/5\Z)$, and for any $v\in (\Z/25\Z)^2$ of order $25$ such that $[G:G_v]= 5$ we have that $G/N_{G}(G_v)\simeq \cC_5$. Therefore we have deduced that if $E/\Q$ is an elliptic curve such that $G_E(5)$ is labeled by \texttt{5B.1.1} and there exists a quintic number field $K$ with a $K$-rational point of order $25$, then $K$ is Galois. Note that in this case there does not exist a point of order $5^n$ for $n>2$ over any quintic number field: suppose that $K'$ is a quintic number field such that there exists $P\in E(K')[5^n]$. Then $[5^{n-2}]P\in E(K')[25]$. Therefore $K'$ is Galois and, by Lemma \ref{lem_chou}, $E$ has a rational $5^n$-isogeny. In contradiction with Theorem \ref{isog}. This completes the proof of (\ref{i1_lem5}).

Finally we assume that $G_E(5)$ is labeled by \texttt{5Cs.1.1}. That is, $G_E(5)$ is conjugate in $\GL_2(\Z/5\Z)$ to the subgroup (cf. \cite[Theorem 1.4 (iii)]{zywina})
$$
H_{1,1}=\left\langle\begin{pmatrix}1 &0 \\ 0 & 2\end{pmatrix}\right\rangle\subset \GL_2(\Z/5\Z).
$$
In this case using a similar algorithm as above  \href{http://www.uam.es/personal_pdi/ciencias/engonz/research/tables/tors5/Lemma10.txt}{\color{blue}{we check}} that if there exists a quintic number field $K$ such that $E(K)[25]\simeq \cC_{25}$ then $K$ is Galois or the Galois closure of $K$ in $\Qbar$ is isomorphic to $\mathcal F_5$, where $\mathcal F_5$ denotes the Fr\"obenius group of order $20$. In the former case, this proves that there does not exist a point of order $5^n$ for $n>2$ over any Galois quintic number field. Now, assume that $K$ is not Galois, then $G_E(125)$ satisfies: 
$$
\begin{array}{lcl}
G_E(125)\equiv\, G_E(5) \,\,(\mbox{mod $5$}) & ,& [G_E(125):H_P]=5, \\
G_E(125)\equiv\, G_E(25) \,\,(\mbox{mod $25$}) &, &  [G_E(25):H_{5P}]=5.
\end{array}
$$
 \href{http://www.uam.es/personal_pdi/ciencias/engonz/research/tables/tors5/Lemma10.txt}{\color{blue}{We check}} that for any subgroup $G$ of $GL_2(\Z/125\Z)$ satisfying $G \equiv H\,\, (\mbox{mod $5$})$ for some conjugate $H$ of $H_{1,1}$ in $\GL_2(\Z/5\Z)$, and for any $v\in (\Z/125\Z)^2$ of order $125$ such that $[G:G_v]= 5$ and $G/N_{G}(G_v)\simeq \mathcal F_5$ we obtain that $[G':G'_w]\ne 5$ for any $w\in (\Z/25\Z)^2$ of order $25$; where $G'\equiv\, G \,\,(\mbox{mod $25$})$. We deduce that there do not exist points of order $125$ over quintic number fields. So, $E(K)[5^\infty]\lesssim \cC_{25}$. 

\

This finishes the proof.

\end{proof}

\subsection{$11$-primary torsion subgroup} 

\begin{lemma}\label{lem11}
Let $E/\Q$ be an elliptic curve and $K/\Q$ a quintic number field. Then 
$$
E(K)[11^\infty]\lesssim \cC_{11}.
$$
In particular, if $E(K)[11^\infty]\ne \{\mathcal O\}$ then $E$ is labeled \texttt{121a2}, \texttt{121c2}, or \texttt{121b1}, $K=\Q(\zeta_{11})^+$ and $E(K)_{\tors}\simeq \cC_{11}$. 
\end{lemma}
\begin{proof}
First, suppose that $E/\Q$ is non-CM. Then Table \ref{tableSutherland} shows that there exists a point of order $11$ over a quintic number field if and only if $G_E(11)$ is labeled \texttt{11B.1.4} or \texttt{11B.1.5}. Or in Zywina notation, $G_E(11)$ is conjugate in $\GL_2(\Z/11\Z)$ to the subgroups $H_{1,1}$ or $H_{2,1}$. Then Zywina  \cite[Theorem 1.6(v)]{zywina} proved that 
$E$ is isomorphic (over $\Q$) to \texttt{121a2} or \texttt{121c2} respectively.

Now, let us suppose that $E/\Q$ has CM. Recall  that there are thirteen $\Q$-isomorphic classes of elliptic curve with CM (cf. \cite[A \S 3]{advanced}), each of them has CM by an order in the imaginary quadratic field with discriminant $-D$, where $D\in \{3,4,7,8,11,19,43,67,163\}$.  In this context, Zywina  \cite[\S 1.9]{zywina} gives a complete characterization of the conjugacy class of $G_E(p)$ in $\GL_2(\Z/p\Z)$, for any prime $p$. Let us apply these results for the case $p=11$.  The proof splits on whether $j(E)\ne 0$ (Proposition 1.14 \cite{zywina}) or $j(E)= 0$ (Proposition 1.16 (iv) \cite{zywina}): 
\begin{itemize}
\item $j(E)\ne 0$.  Depending whether $-D$ is a quadratic residue modulo $11$:
\begin{itemize}
\item if $D\in\{7,8,19,43\}$ then $G_E(11)$ is conjugate to \texttt{11Ns}.
\item if $D\in\{3,4,6,7,163\}$ then $G_E(11)$ is conjugate to \texttt{11Nn}.
\item if $D=11$:
\begin{itemize}
\item  if $E$ is \texttt{121b1} then $G_E(11)$ is conjugate to \texttt{11B.1.3},
\item  if $E$ is \texttt{121b2} then $G_E(11)$ is conjugate to \texttt{11B.1.8},
\item otherwise $G_E(11)$ is conjugate to \texttt{11B.10.3}.
\end{itemize}  
\end{itemize}  
\item $j(E)=0$. Then $G_E(11)$ is conjugate to \texttt{11Nn.1.4} or \texttt{11Ns}.    
\end{itemize}
The following \href{http://www.uam.es/personal_pdi/ciencias/engonz/research/tables/tors5/Lemma11.txt}{\color{blue}table} lists for each possible $G_E(11)$ as above, the value $d_1$, the minimum of the indexes of the stabilizers of $v\in(\Z/11\Z)^2$, $v\ne (0,0)$, by the action of $G_E(11)$ on $(\Z/11\Z)^2$; equivalently, the minimum degree of the extension $L/\Q$ over which $E$ has a $L$-rational point of order $11$.
$$
\begin{tabular}{|c|c|c|c|c|c|}
\hline
\texttt{11Ns} & \texttt{11Nn} & \texttt{11B.1.3} & \texttt{11B.1.8} &  \texttt{11B.10.3} &  \texttt{11Nn.1.4}\\
\hline
20 & 120 & 5 & 10 & 10& 40\\
\hline
\end{tabular}
$$
The above table proves that $E/\Q$ has a point of order $11$ over a quintic number fields if and only if $E$ is the curve \texttt{121b1}.

Finally, Table \ref{ex_5} shows that the torsion of the elliptic curves \texttt{121a2}, \texttt{121c2} and \texttt{121b1} grows in a quintic number field to $\cC_{11}$ only over the field $\Q(\zeta_{11})^+$, and over that field the torsion is $\cC_{11}$. 
\end{proof}

\begin{remark}
If in the above statement the quintic number field is replaced by a number field $K$ of degree $d$ such that $d\ne 5$ and $d\le 9$, then there does not exist any elliptic curve $E/\Q$ with a point of order $11$ over $K$. 

\end{remark}
\subsection{$\{p,q\}$-primary torsion subgroup} 

\begin{lemma}\label{lem_pq}
Let $E/\Q$ be an elliptic curve and $K/\Q$ a quintic number field. Let $p,q\in \{2,3,5,7,11\}$, $p\ne q$, such that $pq$ divides the order of $E(K)_{\tors}$. Then
$$
E(\Q)[\{p,q\}^\infty] =E(K)[\{p,q\}^\infty] \quad\mbox{or}\quad E(K)[\{p,q\}^\infty] \simeq \cC_{10}.
$$
In the former case, $E(\Q)_{\tors}=E(\Q)[\{p,q\}^\infty]\simeq G$, where $G\in\{\cC_6,\cC_{10},\cC_{2}\times\cC_{6}\}$.

\end{lemma}
 
\begin{proof}
First we may suppose $p\ne 11$ by Lemma \ref{lem11}.  Assume that $p,q\in\{2,3,7\}$, then by Lemma \ref{lem_p} we have that the $\{p,q\}$-primary torsion is defined over $\Q$. That is, $E(K)[\{p,q\}^\infty] =E(\Q)[\{p,q\}^\infty]$. Let $G\in \Phi(1)$ such that $E(\Q)_{\tors} \simeq G$. Then $G\in\{\cC_6,\cC_{2}\times\cC_{6}\}$.

It remains to prove the case $p=5$ and $q\in\{2,3,7\}$. Without loss of generality we can assume that the $5$-primary torsion is not defined over $\Q$, otherwise $E(K)[\{5,q\}^\infty] =E(\Q)[\{5,q\}^\infty]$ and the unique possibility is $\cC_{10}$. In particular, by Lemma \ref{lem5} we have that $E$ has non-CM and the $5$-primary torsion of $E$ over $K$ is cyclic of order $5$ or $25$, and $E(\Q)[5^\infty]=\{\mathcal O\}$ or $E(\Q)[5^\infty]\simeq \cC_5$ respectively. Depending on $q\in\{2,3,7\}$ we have: 
\\
$\bullet$ $q=2$:
  \begin{itemize}
\item[$\star$] $E(K)[5^\infty]\simeq \cC_5$. If $E(K)[2^\infty]\simeq \cC_2$ then there are infinitely many elliptic curves such that $E(K)[\{2,5\}^\infty] \simeq \cC_{10}$ (see Proposition \ref{inf_5-10}). In fact, the above $2$-primary torsion is the unique possibility since if $\cC_4\lesssim E(\Q)$ then $\cC_{20}\not \lesssim E(K)$ and if $E[2]\lesssim E(\Q)$ then $\cC_{2}\times \cC_{10}\not \lesssim E(K)$ (see Remark below Theorem 7 of \cite{GL16}).
\item[$\star$]  $E(K)[5^\infty]\simeq \cC_{25}$. Assume that $E(K)[2]\ne \{\mathcal O\}$. If $G_E(5)$ is labeled \texttt{5B.1.1} then $K$ is Galois and therefore, by Lemma \ref{lem_chou}, $E$ has a rational $50$-isogeny, that is not possible by Theorem \ref{isog}. Now suppose that $G_E(5)$ is labeled \texttt{5Cs.1.1}. Since $E(K)[2^\infty]=E(\Q)[2^\infty]$ and  $E(\Q(\zeta_5))=E[5]$ (by Table \ref{tableSutherland}) we deduce $\cC_5\times \cC_{10}\lesssim E(\Q(\zeta_5))$. But this is not possible since Bruin and Najman \cite[Theorem 6]{BN16} have proved that any elliptic curve defined over $\Q(\zeta_5)$ have torsion subgroup isomorphic to a group in the following set
$$
\Phi(\Q(\zeta_5))=\left\{ \, \cC_n \; | \; n=1,\dots,10,12,15,16 \right\} \cup \left\{\, \cC_2 \times \cC_{2m} \; | \; m=1,\dots,4 \right\}\cup\{\cC_5\times\cC_5\}.
$$
\end{itemize}
$\bullet$ $q=3$: A necessary condition if $15$ divides $E(K)_{\tors}$ is that the $5$-torsion is not defined over $\Q$ and the $3$-torsion is defined over $\Q$. By Lemma \ref{lem5}, $G_E(5)$ is labeled \texttt{5B.1.2}. Zywina \cite[Theorem 1.4]{zywina}  has showed that  its $j$-invariant is of the form
$$
J_5(t)=\frac{(t^4+228t^3+494t^2-228t+1)^3}{t(t^2-11t-1)^5},\qquad\mbox{for some $t\in \Q$.}
$$
On the other hand, we have proved that the $3$-torsion is defined over $\Q$. Then, by Table \ref{tableSutherland}, $G_E(3)$ is labeled \texttt{3Cs.1.1} or \texttt{3B.1.1}. Again Zywina \cite[Theorem 1.2]{zywina} characterizes the $j$-invariant of $E/\Q$ depending on the conjugacy class of $G_E(3)$: 
  \begin{itemize}
\item[$\star$] \texttt{3Cs.1.1}: $\displaystyle J_1(s)=27\frac{(s+1)^3(s+3)^3(s^2+3)^3}{s^3(s^2+3s+3)^3}$, for some $s\in \Q$. We must have an equality of $j$-invariants: $J_1(s)=J_5(t)$. In particular, grouping cubes we deduce:
$$
t(t^2-11t-1)^2=r^3,\qquad\mbox{for some $t,r\in\Q$}.
$$
This equation defines a curve $C$ of genus $2$, which in fact transforms (\href{http://www.uam.es/personal_pdi/ciencias/engonz/research/tables/tors5/Lemma12.txt}{\color{blue}{according to \texttt{Magma}}}) to\footnote{A remarkable fact is that this genus $2$ curve is {\it new modular} of level $45$ (see \cite{EGJ}).} $C'\,:\,y^2 = x^6 + 22x^3 + 125$. The jacobian of $C'$ has rank $0$, so we can use the Chabauty method, and determine that the points on $C'$ are 
$$
C'(\Q)=\{ (1 : \pm 1: 0)\}.
$$
Therefore $C'$ has no affine points and we obtain
$$
C(\Q)=\{ (0,0) \}\cup\{(1:0:0)\}.
$$
Then $t=0$, and since $t$ divides the denominator of $J_5(t)$ we have reached a contradiction to the existence of such curve $E$.\\
\item[$\star$] \texttt{3B.1.1}: $\displaystyle  J_3(s)=27\frac{(s+1)(s+9)^3}{s^3}$, for some $s\in \Q$. A \href{http://www.uam.es/personal_pdi/ciencias/engonz/research/tables/tors5/Lemma12.txt}{\color{blue}{similar argument}} with the equality $J_3(s)=J_5(t)$ gives us the equation:
$$
C\,:\,27(s+1)(s+9)^3t(t^2-11t-1)^5=s^3(t^4+228t^3+494t^2-228t+1)^3.
$$
In this case the above equation defines a genus $1$ curve which has the following points:
{\small
$$
\{\left({-2}/{27} , {-1}/{8} \right),\left({-27}/{2} , -2\right), \left({-27}/{2} , {1}/{2} \right), \left(0 , 0 \right), \left({-2}/{27} , 8\right)\}\cup \{(0 : 1 : 0),\left({1}/{27} : 1 : 0\right),(1 : 0 : 0)
\}.
$$
The curve $C$ is $\Q$-isomorphic to the elliptic curve \texttt{15a3}, which Mordell-Weil group (over $\Q$) is of order $8$. }Therefore we deduce that $s={-2}/{27}, {-27}/{2}$, and in particular
$$
j(E)\in \{-5^2/2,-5^2\cdot 241^3/2^3\}.
$$
Therefore there are two $\Qbar$-isomorphic classes of elliptic curves. Each pair of elliptic curves in the same $\Qbar$-isomorphic class is related by a quadratic twist. Najman \cite{twist} has made an exhaustive study of how the torsion subgroup changes upon quadratic twists. In particular Proposition 1 (c) \cite{twist} asserts that if $E/\Q$ is neither \texttt{50a3} nor \texttt{450b4}, and it satisfies $E(\Q)_{\tors}\simeq \cC_3$ and the $(-3)$-quadratic twist $E^{-3}$, satisfies $E^{-3}(\Q)_{\tors}\not \simeq \cC_3$, then for any quadratic twist we must have $E^d(\Q)\simeq \cC_1$ for all $d\in \Q^*/(\Q^*)^2$. We apply this result to the elliptic curves \texttt{50a1} and \texttt{450b2} that have $j$-invariant $-5^2/2$ and $-5^2\cdot 241^3/2^3$ respectively. Both curves have cyclic torsion subgroup (over $\Q$) of order $3$ and the corresponding torsion subgroup of the $(-3)$-quadratic twist is trivial. Thus we are left with two elliptic curves (\texttt{50a1} and \texttt{450b2}) to finish the proof. Applying the algorithm described in Section \ref{sec_ex} we compute that the $5$-torsion does not grow over any quintic number field for both curves.
\end{itemize}

$\bullet$ $q=7$. Similar to the the case $q=3$, we deduce that $E/\Q$ has the $7$-torsion defined over $\Q$ and $G_E(5)$ is labeled \texttt{5B.1.2}. Looking at Table \ref{tableSutherland} we deduce that $E/\Q$ has a rational $5$-isogeny, since $d_0=1$ for \texttt{5B.1.2}. Then, since $E/\Q$ has a point of order $7$ defined over $\Q$, there exists a rational $35$-isogeny, which contradicts Theorem \ref{isog}. 
\end{proof}

\subsection{$\{p,q,r\}$-primary torsion subgroup} 

\begin{lemma}
Let $E/\Q$ be an elliptic curve and $K/\Q$ a quintic number field. Let $p,q,r\in \{2,3,5,7,11\}$, $p\ne q\ne r$, such that $pqr$ divides the order of $E(K)_{\tors}$. Then $E(K)[\{p,q,r\}^\infty]=\{\mathcal O\}$.
\end{lemma}
 
\begin{proof}
Lemma \ref{lem_pq} shows that there do not exist three different primes $p,q,r$ such that $pqr$ divides the order of $E(K)_{\tors}$.
\end{proof}

\section{Proof of Theorems \ref{main1}, \ref{main2} and \ref{main3}}

We are ready to prove Theorems   \ref{main1}, \ref{main2} and \ref{main3}. 

\begin{proof}[Proof of Theorem \ref{main1}]
Since we have $\Phi_\Q(1)\subseteq \Phi_\Q(5)$, let us prove that the unique torsion structures that remain to add to $\Phi_\Q(1)$ to obtain $\Phi_\Q(5)$ are $\cC_{11}$ and $\cC_{25}$. Let $H\in \Phi_\Q(5)$ be such that $H\not\in\Phi_\Q(1)$. Lemma \ref{lem_pq} shows that $|H|=p^n$, for some prime $p$ and a positive integer $n$. Now, Lemma \ref{lem_p} shows that $p\in\{5,11\}$. If $p=11$ then $n=1$ by Lemma \ref{lem11}. If $p=5$ then $n=2$ by Lemma \ref{lem5}, and an example with torsion subgroup isomorphic to $\cC_{25}$ is given in Table \ref{ex_5}. This finish the proof for the set $\Phi_\Q(5)$.

Now the CM case.  Notice that $\Phi^{\cm}_\Q(1)\subseteq \Phi^{\cm}_\Q(5) \subseteq \Phi^{\cm}(5)$. We have that the unique torsion structure that belongs to $\Phi^{\cm}(5)$ and not to $\Phi^{\cm}_\Q(1)$ is $\cC_{11}$. But in Lemma \ref{lem11} we have proved that the elliptic curve \texttt{121b1} has torsion subgroup isomorphic to $\cC_{11}$ over $\Q(\zeta_{11})^+$.  Therefore $\Phi^{\cm}_\Q(5) = \Phi^{\cm}(5)$. This finishes the proof.
\end{proof}

The determination of $\Phi_\Q(5,G)$ will rest on the following result:

\begin{proposition}\label{propG}
Let $E/\Q$ be an elliptic curve and $K/\Q$ a quintic number field such that $E(\Q)_{\tors}\simeq G$ and $E(K)_{\tors}\simeq H$.
\begin{enumerate}
\item\label{t1} Let $p\in\{2,3,7\}$ and $G$ of order a power of $p$, then $H=G$.
\item\label{t25} If $H=\cC_{25}$, then $G=\cC_5$.
\end{enumerate}
\end{proposition}
\begin{proof}
The item (\ref{t1}) follows from Lemma \ref{lem_p} and (\ref{t25}) from Lemma \ref{lem5} (\ref{i1_lem5}). 
\end{proof}

\begin{proof}[Proof of Theorem \ref{main2}]

Let $E/\Q$ be an elliptic curve and $K/\Q$ a quintic number field such that 
$$
E(\Q)_{\tors}\simeq G\qquad\mbox{and}\qquad E(K)_{\tors}\simeq H.
$$
The group $H\in\Phi_\Q(5)$ (row in Table \ref{table1}) that does not appear in some $\Phi_\Q(5,G)$ for any $G\in\Phi(1)$  (column in Table \ref{table1}), with $G\subseteq H$  can be ruled out using Proposition \ref{propG}. In Table \ref{table1} we use:
\begin{itemize}
\item (\ref{t1}) and (\ref{t25}) to indicate which part of Proposition \ref{propG} is used, 
\item the symbol $-$ to mean the case is ruled out because $G \not\subset H$, 
\item with a $\checkmark$, if the case is possible and, in fact, it occurs. There are two types of check marks in Table \ref{table1}:
\begin{itemize}  
\item $\checkmark$ (without a subindex) means that $G=H$.
\item  $\checkmark_{\!\!\!5}$ means that $H\ne G$ can be achieved over a quintic number field $K$, and we have collected examples of curves and quintic number fields in Table \ref{ex_5}.
\end{itemize}
\end{itemize}

\

{\footnotesize
\renewcommand{\arraystretch}{1.2}
\begin{longtable}[h]{|c|c|c|c|c|c|c|c|c|c|c|c|c|c|c|c|}
\cline{1-16}
\backslashbox{$H$}{$G$} & $\cC_1$ & $\cC_2$ & $\cC_3$ & $\cC_4$ & $\cC_5$ & $\cC_6$ & $\cC_7$ & $\cC_8$ & $\cC_9$ & $\cC_{10}$ & $\cC_{12}$ & $\cC_2 \times \cC_2$ & $\cC_2 \times \cC_4$& $\cC_2 \times \cC_6$& $\cC_2 \times \cC_8$\\
\hline
\endfirsthead
\cline{1-16}
\backslashbox{$H$}{$G$} & $\cC_1$ & $\cC_2$ & $\cC_3$ & $\cC_4$ & $\cC_5$ & $\cC_6$ & $\cC_7$ & $\cC_8$ & $\cC_9$ & $\cC_{10}$ & $\cC_{12}$ & $\cC_2 \times \cC_2$ & $\cC_2 \times \cC_4$& $\cC_2 \times \cC_6$& $\cC_2 \times \cC_8$\\
\hline
\endhead
\endfoot
\endlastfoot
$\cC_1$ & $\checkmark$ & $-$ & $-$ & $-$ & $-$ & $-$ & $-$ & $-$ & $-$ & $-$ & $-$ & $-$ & $-$ & $-$ & $-$ \\
\hline
$\cC_2$ & (\ref{t1}) & $\checkmark$ & $-$ & $-$ & $-$ & $-$ & $-$ & $-$ & $-$ & $-$ & $-$ & $-$ & $-$ & $-$ & $-$ \\
\hline
$\cC_3$ & (\ref{t1}) & $-$ & $\checkmark$ & $-$ & $-$ & $-$ & $-$ & $-$ & $-$ & $-$ & $-$ & $-$ & $-$ & $-$ & $-$ \\
\hline
$\cC_4$ & (\ref{t1}) & (\ref{t1}) & $-$ & $\checkmark$ & $-$ & $-$ & $-$ & $-$ & $-$ & $-$ & $-$ & $-$ & $-$ & $-$ & $-$\\
\hline
$\cC_5$ & $\checkmark_{\!\!\!5}$ & $-$ & $-$ & $-$ & $\checkmark$ & $-$ & $-$ & $-$ & $-$ & $-$ & $-$ & $-$ & $-$ & $-$ & $-$ \\
\hline
$\cC_6$ & (\ref{t1}) & (\ref{t1}) & (\ref{t1}) & $-$ & $-$ & $\checkmark$ & $-$ & $-$ & $-$ & $-$ & $-$ & $-$ & $-$ & $-$ & $-$\\
\hline
$\cC_7$ & (\ref{t1})  & $-$ & $-$ & $-$ & $-$ & $-$ & $\checkmark$ & $-$ & $-$ & $-$ & $-$& $-$ & $-$ & $-$ & $-$ \\
\hline
$\cC_8$ &   (\ref{t1}) & (\ref{t1}) & $-$ & (\ref{t1}) & $-$ & $-$ & $-$ &  $\checkmark$ & $-$ & $-$ & $-$  & $-$ & $-$ & $-$ & $-$ \\
\hline
$\cC_9$ & (\ref{t1}) & $-$ & (\ref{t1}) & $-$ & $-$ & $-$ & $-$ & $-$ & $\checkmark$ & $-$ & $-$& $-$ & $-$ & $-$ & $-$\\
\hline
$\cC_{10}$ & (\ref{t1}) &  $\checkmark_{\!\!\!5}$ & $-$ & $-$ &   (\ref{t1})& $-$ & $-$ & $-$ & $-$ & $\checkmark$ & $-$ & $-$ & $-$ & $-$ & $-$ \\
\hline
$\cC_{11}$ & $\checkmark_{\!\!\!5}$ & $-$ & $-$ & $-$ & $-$ & $-$ & $-$ & $-$ & $-$ & $-$ & $-$ & $-$ & $-$ & $-$ & $-$ \\\hline
$\cC_{12}$ &  (\ref{t1}) &  (\ref{t1}) &  (\ref{t1}) &  (\ref{t1}) & $-$ &  (\ref{t1}) & $-$ & $-$ & $-$ & $-$ & $\checkmark$ & $-$ & $-$ & $-$ & $-$\\
\hline
$\cC_{25}$ & (\ref{t25})  & $-$ & $-$ & $-$ & $\checkmark_{\!\!\!5}$ & $-$ & $-$ & $-$ & $-$ & $-$ & $-$  & $-$ & $-$ & $-$ & $-$\\
\hline
$\cC_2 \times \cC_2$ & (\ref{t1})  & (\ref{t1})   & $-$ & $-$ & $-$ & $-$ & $-$ & $-$ & $-$ & $-$ & $-$  & $\checkmark$ & $-$ & $-$ & $-$\\
\hline
$\cC_2 \times \cC_4$ & (\ref{t1})  & (\ref{t1})  & $-$ & (\ref{t1})  & $-$ & $-$ & $-$ & $-$ & $-$ & $-$ & $-$ & (\ref{t1})  & $\checkmark$ & $-$ & $-$\\
\hline
$\cC_2 \times \cC_6$ &(\ref{t1})   & (\ref{t1})  & (\ref{t1})  & $-$ & $-$ & (\ref{t1})  & $-$ & $-$ & $-$ & $-$ & $-$ & (\ref{t1})  & $-$ & $\checkmark$ & $-$\\
\hline
$\cC_2 \times \cC_8$ & (\ref{t1})  & (\ref{t1})  & $-$ & (\ref{t1})  & $-$ & $-$ & $-$ &  (\ref{t1}) & $-$ & $-$ & $-$ &(\ref{t1})  & (\ref{t1})  & $-$ &$\checkmark$ \\
\hline
\caption{The table displays either if the case happens for $G=H$ ($\checkmark$), if it occurs over a quintic ($\checkmark_{\!\!\!5}$), if it is impossible because $G \not\subset H$ ($-$) or if it is ruled out by Proposition \ref{propG} (\ref{t1}) and (\ref{t25}).}\label{table1}
\end{longtable}
}
It remains to prove that  there are infinitely many $\Qbar$-isomorphism classes of elliptic curves $E/\Q$ with $H\in \Phi_\mathbb{Q} \left(5,G \right)$, except for the case $H=\cC_{11}$. Note that for any elliptic curve $E/\Q$ with $E(\Q)_\text{tors}$, there is always an extension $K/\Q$ of degree $5$ such that $E(K)_\text{tors}=E(\Q)_\text{tors}$. Then for any $G\in \Phi(1) \cap \Phi_\Q(5)$ the statement is proved. Now, since $\Phi_\Q(5)\setminus \Phi(1)=\{\cC_{11},\cC_{25}\}$, the only case that remains to prove is $H=\cC_{25}$. This case will be proved in Proposition \ref{inf_25}.
\end{proof}

\begin{proof}[Proof of Theorem \ref{main3}]
Let $E/\Q$ be an elliptic curve such that the torsion grows to $\cC_{11}$ over a quintic number field $K$. Then by  Lemma \ref{lem11} we know that $K=\Q(\zeta_{11})^+$ and the torsion does not grow for any other quintic number field. Therefore to finish the proof it remains to prove that there does not exist an elliptic curve $E/\Q$ and two non-isomorphic quintic number fields $K_1,K_2$ such that $E(K_i)_{\tors}\simeq H\in \Phi_\Q(5)$, $i=1,2$, and $E(\Q)_{\tors}\not\simeq H$. Note that the compositum $K_1K_2$ satisfies $[K_1K_2:\Q]\le [K_1:\Q][K_2:\Q]=25$. Now, by Theorem \ref{main2} we deduce $H\in \{\cC_5,\cC_{10},\cC_{25}\}$:

$\bullet$ First suppose that $H\in \{\cC_5,\cC_{10}\}$. Then by Lemma \ref{lem5}, $G_E(5)$ is labeled \texttt{5B.1.2}. Now, since $K_1\not\simeq K_2$ we deduce $K_1K_2=\Q(E[5])$ and, in particular, $\Gal(\widehat{K_1K_2}/\Q)\simeq G_E(5)$. In this case we have that $G_E(5)\simeq \mathcal F_5$, where $\mathcal F_5$ denotes the Fr\"obenius group of order $20$. Diagram \ref{F5} shows the lattice subgroup of $\mathcal F_5$, where $\mathcal H_{k,i}$ denotes the $k$-th subgroup of index $i$ in $\mathcal F_5$. Note that all the index $5$ subgroups $\mathcal H_{k,5}$ are conjugates in $\mathcal F_5$. That is, their associated fixed quintic number fields are isomorphic. This proves that $K_1\simeq K_2$.
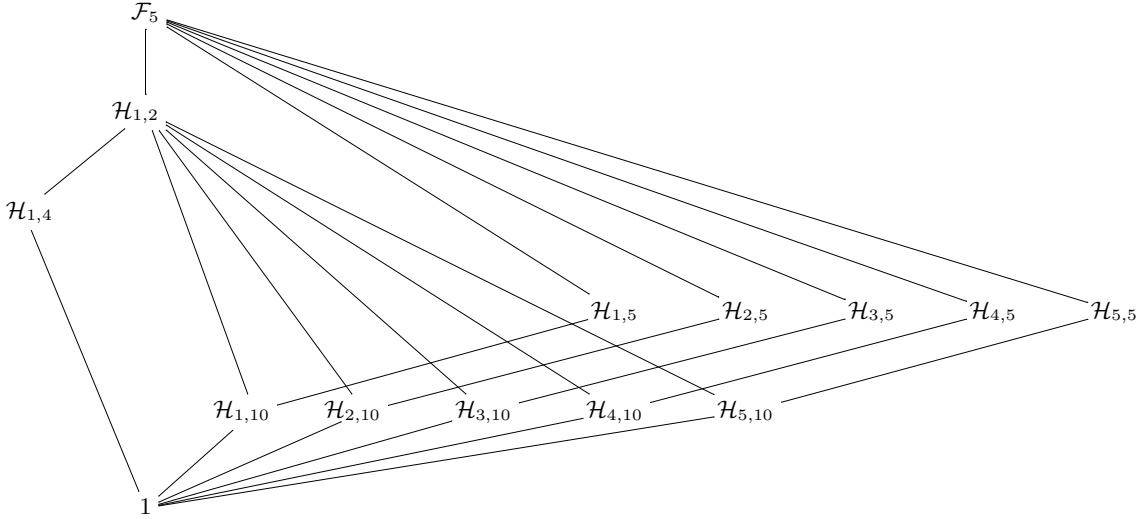
\begin{figure}[h]
{\footnotesize
$$
\xymatrix{
    &    \mathcal F_5     \ar@{-}[d]  \ar@{-}[dddrrrr] \ar@{-}[dddrrrrr] \ar@{-}[dddrrrrrr] \ar@{-}[dddrrrrrrr] \ar@{-}[dddrrrrrrrr] & &&&&&&&&   \\
    &   \!\! \!\!\!\mathcal H_{1,2}    \ar@{-}[dl]  \ar@{-}[dddr] \ar@{-}[dddrrr]  \ar@{-}[dddrr] \ar@{-}[dddrrrr] \ar@{-}[dddrrrrr] & &&&&&&&&   \\
\,\,\,\mathcal H_{1,4}  \ar@{-}[dddr]   &       & &&&&&&&&   \\
   &       & &&& \!\mathcal H_{1, 5}  \ar@{-}[dlll] & \!\mathcal H_{2,5} \ar@{-}[dlll]& \!\mathcal H_{3,5} \ar@{-}[dlll]& \!\mathcal H_{4,5} \ar@{-}[dlll]& \!  \mathcal H_{5,5} \ar@{-}[dlll]\\
   &       & \!\!\!\!\!\!\mathcal H_{1,10}  \ar@{-}[dl]&\!\!\!\!\!\!\mathcal H_{2,10}  \ar@{-}[dll]&\!\mathcal H_{3,10}  \ar@{-}[dlll]&\!\mathcal H_{4,10} \ar@{-}[dllll] &\!\mathcal H_{5,10} \ar@{-}[dlllll]  &&& \\ 
    &    1  & &  &&&&&&&    
    }
$$
}
 \caption{Lattice subgroup of $\mathcal F_5$}\label{F5}
 \end{figure}

$\bullet$ Finally suppose that $H=\cC_{25}$. In this case we use a similar argument as above but replacing $G_E(5)$ by $G_E(25)$. We know by Lemma \ref{lem5} that $G_E(5)$ is labeled \texttt{5B.1.1} or \texttt{5Cs.1.1}, but we do not have an explicit description of $G_E(25)$. For that reason we apply an analogous algorithm as the one used in the proof of Lemma \ref{lem5} (\ref{i1_lem5}). By \cite[Theorem 1.4 (iii)]{zywina}) we have that $G_E(5)$ is conjugate in $\GL_2(\Z/5\Z)$ to
$$
H_{6,1}=\left\langle\begin{pmatrix}1 &0 \\ 0 & 2\end{pmatrix},\begin{pmatrix}1 &1 \\ 0 & 1\end{pmatrix}\right\rangle\quad\mbox{or}\quad H_{1,1}=\left\langle\begin{pmatrix}1 &0 \\ 0 & 2\end{pmatrix}\right\rangle\,,
$$
depending if $G_E(5)$ is labeled \texttt{5B.1.1} or \texttt{5Cs.1.1} respectively.

Suppose that $K_1\not\simeq K_2$, then $K_1K_2=\Q(E[25])$. Therefore $\Gal(\widehat{K_1K_2}/\Q)\simeq G_E(25)$ and $|G_E(25)|\le 25$.  Now, we fix $\mathcal H$ to be $H_{6,1}$ or $H_{1,1}$ and since we do not have an explicit description of $G_E(25)$ \href{http://www.uam.es/personal_pdi/ciencias/engonz/research/tables/tors5/Theorem3.txt}{\color{blue}{we run a \texttt{Magma} program}} where the input is a subgroup $G$ of $GL_2(\Z/25\Z)$ satisfying 
\begin{itemize}
\item $|G|\le 25$,
\item $G \equiv H\,\, (\mbox{mod $5$})$ for some conjugate $H$ of $\mathcal H$ in $\GL_2(\Z/5\Z)$,
\item there exists $v\in (\Z/25\Z)^2$ of order $25$ such that $[G:G_v]=5$.
\end{itemize}
If $\mathcal H=H_{6,1}$ the above algorithm does not return any subgroup $G$. In the case $\mathcal H=H_{1,1}$ all the subgroups returned are isomorphic either to $\mathcal F_5$ or to $\mathcal C_{20}$. If $G\simeq \mathcal F_5$ then we have proved that it has five index $5$ subgroups, all of them at the same conjugation class. If $G\simeq \mathcal C_{20}$ there is only one subgroup of index $5$.   We have reached a contradiction with $K_1\not\simeq K_2$. This finishes the proof.
\end{proof}

\section{Infinite families of rational elliptic curves where the torsion grows over a quintic number field.}
Let $E/\Q$ be an elliptic curve and $K$ a quintic number field such that $E(\Q)_{\tors}\simeq G\in\Phi(1)$ and $E(K)_{\tors}\simeq H\in \Phi_\Q(5)$. Theorem \ref{main3} shows that $G\not\simeq H$ in the following cases:
$$
(G,H)\in\{\,(\cC_1,\cC_5)\,,\,(\cC_1,\cC_{11})\,,\,(\cC_2,\cC_{10})\,,\,(\cC_5,\cC_{25})\,\}.
$$
By Lemma \ref{lem11} we have that the pair $(\cC_1,\cC_{11})$ only occurs in three elliptic curves. For the rest of the above pairs we are going to prove that there are infinitely many non--isomorphic classes of elliptic curves and quintic number fields satisfying each pair.

\subsection{$(\cC_1,\cC_5)$ and $(\cC_2,\cC_{10})$.}\label{sec_61}

Let $E/\Q$ be an elliptic curve and $K$ a quintic number field such that $E(\Q)[5]=\{\mathcal O\}$ and $E(K)[5]\simeq \cC_5$. Then Theorem \ref{main2} tells us that:
$$
E(\Q)_{\tors}\simeq \cC_1\,\,\mbox{and}\,\,  E(K)_{\tors}\simeq \cC_5,\qquad\mbox{or}\qquad E(\Q)_{\tors}\simeq \cC_2\,\,\mbox{and}\,\,  E(K)_{\tors}\simeq \cC_{10}.
$$
First notice that $E$ has non-CM, since $\cC_5$ is not a subgroup of any group in $\Phi^{\cm}(5)$. Then Lemma \ref{lem5} shows that $G_E(5)$ is labeled \texttt{5B.1.2} ($H_{5,1}$ in Zywina's notation). Then Zywina  \cite[Theorem 1.4(iii)]{zywina} proved that there exists $t\in \Q$ such that $E$ is isomorphic (over $\Q$) to $\mathcal E_{5,t}$:
$$
\mathcal E_{5,t}\,:\,  y^2=x^3 -27(t^4 + 228t^3 + 494t^2 - 228t + 1)x+54 (t^6 - 522t^5 - 10005t^4 - 10005t^2 + 522t + 1).
$$
Table \ref{tableSutherland} shows that the degree of the field of definition of a point of order $5$ in $E$ is $4$ or $5$. Moreover, we can \href{http://www.uam.es/personal_pdi/ciencias/engonz/research/tables/tors5/Proposition15.txt}{\color{blue}{compute explicitly}} the number fields factorizing the $5$-division polynomial $\psi_5(x)$ attached to $E$. We define the following polynomial of degree $5$:
$$
{\small
\begin{array}{l}
p_5(x)=x^5 + (-15t^2 - 450t - 15)x^4 + (90t^4 - 65880t^3 + 22860t^2 + 11880t + 90)x^3\\
 \qquad\qquad\quad+ (-270t^6 - 1015740t^5 - 7086690t^4 + 5725080t^3 - 4520610t^2 - 82620t - 270)x^2 \\
 \qquad\qquad\qquad + (405t^8 - 8874360t^7 - 58872420t^6 - 253721160t^5 - 1423822050t^4 + 637175160t^3 + 18109980t^2\\
  \qquad\qquad\qquad\quad + 223560t + 405)x - 243t^{10} - 22886226t^9 - 485812647t^8 + 3223702152t^7 - 34272829350t^6 \\
   \qquad\qquad\qquad\qquad -21920257260t^5 - 53316735462t^4 - 2958964344t^3 - 74726631t^2 - 211410t - 243.
\end{array}
}
$$
Then $p_5(x)$ divides $\psi_5(x)$ and we have $E(\Q(\alpha))[5]=\langle R \rangle\simeq \cC_{5}$, where $p_5(\alpha)=0$ and $\alpha$ is the $x$-coordinate of $R$.

Now suppose that $E(\Q)_{\tors}\simeq \cC_2$, then $G_E(2)$ is labeled \texttt{2B}. Then Zywina \cite[Theorem 1.1]{zywina}  proved that its $j$-invariant is of the form
$$
J_2(s) = 256\frac{(s+1)^3}{s},\qquad\mbox{for some $s\in \Q$.}
$$
Therefore we have $J_2(s)=j(\mathcal E_{5,t})$ for some $s,t\in\Q$. In other words we have a solution of the next equation
$$
256\frac{(s+1)^3}{s}=\frac{(t^4+228t^3+494t^2-228t+1)^3}{t(t^2-11t-1)^5}.
$$
This equation defines a curve $C$ of genus $0$ with $(0,0)\in C(\Q)$, which can be parametrize (\href{http://www.uam.es/personal_pdi/ciencias/engonz/research/tables/tors5/Proposition15.txt}{\color{blue}{according}} to \texttt{Magma} and making a linear change of the projective coordinate in order to simplify the parametrization) by:
$$
(s,t)=\left(\frac{-512 ( 5 r+1) (5 r^2-1)^5}{(5 r-1) ( 5 r+3) ( 
    5 r^2+10 r+1)^5}\,\,,\,\, \frac{2 ( 5 r+3)^2}{(5 r-1)^2 ( 5 r+1)}\right),\qquad\mbox{where $r\in\Q$.}
$$
Finally, replacing the above value for $t$ in $\mathcal E_{5,t}$ and simplifying the Weierstrass equation we obtain: 
$$
E_r:y^2=x^3- 2 (5 r^2+ 2 r + 1) (5 r^4- 40 r^3- 30 r^2+1) x^2+84375 (5 r-1) (5 r+3) (5 r^2+ 10 r+1)^5 x.
$$
Thus we have proved the following result:
\begin{proposition}\label{inf_5-10}
There exist infinitely many $\Qbar$-isomorphic classes of elliptic curves $E/\Q$ such that $E(\Q)_{\tors}\simeq \cC_1$ (resp. $\cC_2$) and infinitely many quintic number fields $K$ such that $E(K)_{\tors}\simeq \cC_{5}$ (resp. $\cC_{10}$).
\end{proposition}
\subsection{$(\cC_5,\cC_{25})$.} Let $E/\Q$ be an elliptic curve such that $G_E(5)$ is labeled by \texttt{5B.1.1} and there exists a quintic number field $K$ with the property $E(K)_{\tors} \simeq \cC_{25}$. Then, by Lemma \ref{lem5} (\ref{i1_lem5}), $K$ is Galois. In particular $E/\Q$ has a rational $25$-isogeny. Then, we observe in \cite[Table 3]{Lozano} that  its $j$-invariant must be of the form:
$$
j_{25}(h)=\frac{(h^{10}+10h^8+35h^6-12h^5+50h^4-60h^3 +25h^2-60h+16)^3}{(h^5+5h^3+5h-11)}, \qquad\mbox{for some $h\in\Q$}.
$$
On the other hand, Zywina \cite[Theorem 1.4(iii)]{zywina} proved that there exists $s\in \Q$ such that $E$ is isomorphic (over $\Q$) to $\mathcal E_{6,s}$:
$$
\mathcal E_{6,s}\,:\,  y^2=x^3-27(s^4 - 12s^3 + 14s^2 + 12s + 1) x+54(s^6-18s^5+75s^4+75s^2+18s+1).
$$
The above $j$-invariants should be equal, so $j(\mathcal E_{6,s})=j_{25}(h)$ for some $s,h\in\Q$. This equality \href{http://www.uam.es/personal_pdi/ciencias/engonz/research/tables/tors5/Proposition16.txt}{\color{blue}{defines}} a non-irreducible curve over $\Q$ whose irreducible components are a genus $16$ curve and a genus $0$ curve.  It is possible to give a parametrization of the above genus $0$ curve such that $s=t^5$, where $t\in\Q$. That is, there exists $t\in\Q$ such that $E$ is $\Q$-isomorphic to $\mathcal E_{6,t^5}$. 

Now, let us define the quintic polynomial $p_{25}(x)$:
$${\small
\begin{array}{l}
p_{25}(x)=x^5 + (-5t^{10} - 12t^8 - 12t^7 - 24t^6 + 30t^5 - 60t^4 + 36t^3 - 24t^2 + 12t - 5)x^4\\
\quad  + (10t^{20} + 48t^{18} + 48t^{17} + 96t^{16} + 24t^{15} + 240t^{14} - 
    144t^{13} + 96t^{12} - 48t^{11} + 236t^{10} + 48t^8 + 48t^7 + 96t^6\\
\quad     - 264t^5 + 240t^4 - 144t^3 + 96t^2 - 48t + 10)x^3 + (-10t^{30} - 72t^{28} - 72t^{27} - 144t^{26} - 
    252t^{25} - 360t^{24} \\
\quad + 216t^{23} - 144t^{22} + 72t^{21} + 1914t^{20} + 720t^{18} + 720t^{17} + 1440t^{16} - 1800t^{15} + 3600t^{14} - 2160t^{13} + 1440t^{12} \\
\quad - 720t^{11} + 1914t^{10} - 72t^8 - 72t^7 - 144t^6 + 612t^5 - 360t^4 + 216t^3 - 144t^2 + 72t - 10)x^2\\
\quad + (5t^{40} + 48t^{38} + 48t^{37} + 96t^{36} + 312t^{35} + 240t^{34} - 144t^{33} + 96t^{32} - 
    48t^{31} - 4516t^{30} - 1584t^{28} - 1584t^{27} \\
    \quad - 3168t^{26} + 19944t^{25} - 7920t^{24} + 4752t^{23} - 3168t^{22} + 1584t^{21} - 18114t^{20} - 1584t^{18} - 1584t^{17} - 3168t^{16} - \\
    \quad 
    12024t^{15} - 7920t^{14} + 4752t^{13} - 3168t^{12} + 1584t^{11} - 4516t^{10} + 48t^8 + 48t^7 + 96t^6 - 552t^5 + 240t^4 - 144t^3\\
    \quad  + 96t^2 - 48t + 5)x - t^{50} - 12t^{48} - 
    12t^{47} - 24t^{46} - 114t^{45} - 60t^{44} + 36t^{43} - 24t^{42} + 12t^{41} + 2371t^{40} \\
    \quad + 816t^{38} + 816t^{37} + 1632t^{36} - 17880t^{35} + 4080t^{34} - 2448t^{33} + 1632t^{32} - 
    816t^{31} + 47294t^{30} - 13896t^{28} \\
    \quad - 13896t^{27} - 27792t^{26} + 34740t^{25} - 69480t^{24} + 41688t^{23} - 27792t^{22} + 13896t^{21} + 47294t^{20} + 816t^{18} + \\
    \quad 816t^{17} + 
    1632t^{16} + 13800t^{15} + 4080t^{14} - 2448t^{13} + 1632t^{12} - 816t^{11} + 2371t^{10} - 12t^8 - 12t^7 - 24t^6 \\
    \quad + 174t^5 - 60t^4 + 36t^3 - 24t^2 + 12t - 1.
\end{array}
}$$
Then $p_{25}(x)$ divides the $25$-division polynomial of $\mathcal E_{6,t^5}$. Fixing $t\in \Q$, \href{http://www.uam.es/personal_pdi/ciencias/engonz/research/tables/tors5/Proposition16.txt}{\color{blue}{we have}} that $\Q(\alpha)/\Q$ is a Galois extension of degree $5$ and $E(\Q(\alpha))=\langle R \rangle\simeq \cC_{25}$, where $p_{25}(\alpha)=0$ and the $x$-coordinate of $R$ is $3\alpha$. Note that $[5]R=(3t^{10} - 18t^5 + 3 , 108t^5)\in E(\Q)$.

We have proved the following result:

\begin{proposition}\label{inf_25}
There exist infinitely many $\Qbar$-isomorphic classes of elliptic curves $E/\Q$ and infinitely many quintic number fields $K$ such that $E(K)_{\tors}\simeq \cC_{25}$. All of them satisfy $E(\Q)_{\tors}\simeq \cC_5$.
\end{proposition}

\subsubsection{A $5$-triangle tale.} Let $E/\Q$ be an elliptic curve such that $G_E(5)$ is labeled by \texttt{5Cs.1.1} ($H_{1,1}$ in Zywina's notation). Zywina \cite[Theorem 1.4(iii)]{zywina} proved that there exists $t\in \Q$ such that $E$ is isomorphic (over $\Q$) to $\mathcal E_{1,t}=\mathcal E_{5,t^5}$. We observe in Table \ref{tableSutherland} that there exists a $\Z/5\Z$-basis  $\{P_1,P_2\}$ of $E[5]$ such that $E(\Q)_{\tors}=\langle P_2 \rangle \simeq \cC_5$,  $E(\Q(\zeta_5))_{\tors}=E[5]=\langle P_1,P_2 \rangle$. Now, since $\langle P_1\rangle$ and $\langle P_2\rangle$ are distinct $\Gal(\Qbar/\Q)$-stable cyclic subgroups of $E(\Qbar)$ of order $5$, there exist two rational $5$-isogenies:
$$
\xymatrix{
    &              E     \ar@{->}[dl]_{\phi_1}  \ar@{->}[dr]^{\phi_2}   &    \\
   E_1 & &  E_2,  }
$$
where the elliptic curves $E_1= E/\langle P_1\rangle$ and $E_2=E/\langle P_2\rangle$ are defined over $\Q$. Using Velu's formulae we can compute explicit equations of these elliptic curves:
$$
\displaystyle  E_1=\mathcal E_{6,t^5},\qquad\qquad  E_2 = \mathcal E_{5,s(t)},\,\,\mbox{where}\,\,\,s(t)=\frac{t(t^4+3\,t^3+4\,t^2+2\,t+1)}{t^4-2\,t^3+4\,t^2-3\,t+1},
$$
Then we have $G_{E_1}(5)$ is labeled by \texttt{5B.1.1} and $G_{E_2}(5)$ is labeled by \texttt{5B.1.2}. We observe that the elliptic curve $E_1$ is the one obtained in the previous section, that is, $E_1(\Q(\alpha))=\langle R \rangle\simeq \cC_{25}$, where $p_{25}(\alpha)=0$ and the $x$-coordinate of $R$ is $3\alpha$. In particular, $E_1$ has a rational $25$-isogeny. Note that $[5]R=Q_2=(3t^{10} - 18t^5 + 3 , 108t^5)$ is such that $E_1(\Q)[5]=\langle Q_2\rangle\simeq \cC_5$ and $E_1(L)[5]=E_1[5]=\langle Q_1,Q_2\rangle$ with $[L:\Q]=20$. If $\widehat{\phi_1}\,:\, E_1 \longrightarrow E$ denotes the dual isogeny of $\phi_1$, then we have $\phi_2\circ\widehat{\phi_1}(\langle R\rangle)=\mathcal O\in E_2$. That is, $\phi_2\circ\widehat{\phi_1},:\, E_2 \longrightarrow E_1$ is a rational $25$-isogeny.

\begin{remark}
\href{http://www.uam.es/personal_pdi/ciencias/engonz/research/tables/tors5/Remark.txt}{\color{blue}{There are}} only seven elliptic curves (\texttt{11a1}, \texttt{550k2}, \texttt{1342c2}, \texttt{33825be2}, \texttt{165066d2}, \texttt{185163a2} and \texttt{192698c2}) with conductor less than $350.000$ such that the corresponding mod $5$ Galois representation is labeled \texttt{5Cs.1.1}. All of them give the corresponding $5$-triangle with the associated elliptic curve (\texttt{11a3}, \texttt{550k3}, \texttt{1342c1}, \texttt{33825be3}, \texttt{165066d1}, \texttt{185163a1} and \texttt{192698c1} resp.) with $\cC_{25}$ torsion over the corresponding quintic number field. Notice that there are no more elliptic curves with conductor less than $350.000$ and torsion isomorphic to $\cC_{25}$ over a quintic number field.
\end{remark}


\section{examples}\label{sec_ex}
Given an elliptic curve $E/\Q$, we describe a method to compute the quintic number field where the torsion could grow. If $E$ is \texttt{121a2}, \texttt{121c2} or \texttt{121b1} we have proved in Lemma \ref{lem11} that the torsion grows to $\cC_{11}$ over the quintic number field $\Q(\zeta_{11})^+$. For the rest of the elliptic curves, we first compute $E(\Q)_{\tors}\simeq G\in\Phi(1)$. If $G\ne \cC_1,\cC_2,\cC_5$, then by Theorem \ref{main2} the torsion remains stable under any quintic extension. If $G=\cC_1$ or $\cC_2$ then, by Theorem \ref{main2}, the torsion could grow to $\cC_5$ or $\cC_{10}$ respectively. Now compute the $5$-division polynomial $\psi_5(x)$. It follows that the quintic number fields where the torsion could grow are contained in the number fields attached to the degree $5$ factors of $\psi_5(x)$. In the case $G=\cC_5$ the torsion could grow to $\cC_{25}$, and the method is similar, replacing the $5$-division polynomial by the $25$-division polynomial. We explain this method with an example.


\begin{example}
Let $E$ be the elliptic curve \texttt{11a2}. We compute $E(\Q)_{\tors}\simeq \cC_1$. Now, the $5$-division polynomial has two degree $5$ irreducible factors: $p_1(x)$ and $p_2(x)$. Let $\alpha_i\in\Qbar$ such that $p_i(\alpha_i)=0$, $i=1,2$. We deduce $\Q(\sqrt[5]{11})=\Q(\alpha_1)=\Q(\alpha_2)$ and $E(\Q(\sqrt[5]{11}))_{\tors}\simeq \cC_5$.
\end{example}

Table \ref{ex_5} shows examples where the torsion grows over a quintic number field. Each row shows the label of an elliptic curve $E/\Q$ such that $E(\Q)_{\tors}\simeq G$, in the first column, and $E(K)_{\tors}\simeq H$, in the second column, and the quintic number field $K$ in the third column.

\begin{table}[h]
\begin{tabular}{|c|c|c|c|}
\hline
$G$ & $H$ & quintic  & label \\
\hline
\multirow{2}{*}{$\cC_{1}$}  & $\cC_{5}$ & $\Q(\sqrt[5]{11})$ & \texttt{11a2} \\ \cline{2-4}
 & $\cC_{11}$ &  $\Q(\zeta_{11})^+$  & \texttt{121a2}\,\,,\,\,\texttt{121c2}\,\,,\,\,\texttt{121b1}\\ \hline
 $\cC_{2}$ & $\cC_{10}$ & $\Q(\sqrt[5]{12})$ & \texttt{  66c3} \\ \hline
 $\cC_{5}$ & $\cC_{25}$ & $\Q(\zeta_{11})^+$ & \texttt{11a3} \\ \hline
 \end{tabular}\\
 \caption{Examples of elliptic curves such that $G\in\Phi(1)$, $H\in\Phi_\Q(5,G)$ and  $G\ne H$.}\label{ex_5}
 \end{table}

\begin{remark}
Note that, although we have proved in Propositions \ref{inf_5-10} and \ref{inf_25} that there are infinitely many elliptic curves over $\Q$ such that the torsion grows over a quintic number field, these elliptic curve seems to appear not very often. \href{http://www.uam.es/personal_pdi/ciencias/engonz/research/tables/tors5/Remark.txt}{\color{blue}{We have computed}} for all elliptic curves over $\Q$ with conductor less than $350.000 $ from \cite{cremonaweb} (a total of $2.188.263$ elliptic curves) and we have found only $1256$ cases where the torsion grows. Moreover, only $40$ cases when it grows to $\cC_{10}$ and $7$ to $\cC_{25}$ (the elliptic curves \texttt{11a3}, \texttt{550k3}, \texttt{1342c1}, \texttt{33825be3} \texttt{165066d1}, \texttt{185163a1} and \texttt{192698c1}).
\end{remark} 

\begin{ack}
We thank \'Alvaro Lozano-Robledo, Filip Najman, and David Zureick-Brown for useful conversations. The author would like to thank the anonymous referee for useful comments and suggestions.
\end{ack}


\begin{thebibliography}{11}

\bibitem{antwerp}
B.J. Birch and W. Kuyk (eds.), {\it Modular Functions of One Variable IV}. Lecture Notes in Mathematics {\bf 476}. Springer (1975).


\bibitem{magma}
W. Bosma, J. Cannon, C. Fieker, and A. Steel (eds.). {\em {H}andbook of {M}agma functions, {E}dition 2.20}. \href{http://magma.maths.usyd.edu.au/magma}{http://magma.maths.usyd.edu.au/magma}, 2015.


\bibitem{BN16}
P. Bruin and F. Najman, {\it A criterion to rule out torsion groups for elliptic curves over number fields}, Res. Number Theory {\bf 2}, 3 (2016).


\bibitem{chou}
M. Chou, {Torsion of rational elliptic curves over quartic Galois number fields}. J. Number Theory {\bf 160} (2016), 603--628.

\bibitem{Clark2014}
P. L.~Clark, P.~Corn, A.~Rice, and J.~Stankewicz.
 \emph{Computation on elliptic curves with complex multiplication}, 
 LMS J. Comput. Math. \textbf{17} (2014), 509--539.
 
\bibitem{cremonaweb}
J.E. Cremona, {\em Elliptic curve data for conductors up to 350.000.} Available on \href{http://johncremona.github.io/ecdata/}{http://johncremona.github.io/ecdata/}, 2015.


\bibitem{DS16}
M. Derickx and A.~V. Sutherland, {\em Torsion subgroups of elliptic curves over quintic and sextic number fields}. Proc. Amer. Math. Soc., to appear.


\bibitem{FSWZ90}
G.~Fung, H.~Str\"oher, H.~Williams, and H.~Zimmer.
{\em Torsion groups of elliptic curves with integral $j$-invariant over pure cubic fields}.
J. Number Theory {\bf 36} (1990) 12--45.

\bibitem{EGJ}
E.~Gonz{\'a}lez-Jim{\'e}nez and J.~Gonz{\'a}lez.
{\em {M}odular curves of genus 2}. 
Math. Comp. {\bf 72} (2003) 397--418.

\bibitem{GL16}
E. Gonz\'alez--Jim\'enez and \'A. Lozano--Robledo, {\em On torsion of rational elliptic curves over quartic fields.}   Math. Comp., to appear. 

\bibitem{GJN16}
E. Gonz\'alez--Jim\'enez and F. Najman, {\em Growth of torsion groups of elliptic curves upon base changes.} 
 \href{http://arxiv.org/abs/1609.02515}{arXiv:1609.02515}.
 
\bibitem{GJNT15}
E. Gonz\'alez--Jim\'enez, F. Najman, and J.M. Tornero, {\em Torsion of rational elliptic curves over cubic fields.}  Rocky Mountain J. Math. {\bf 46} (2016), no. 6, 1899--1917.

\bibitem{GJT14}
E. Gonz\'alez--Jim\'enez and J.M. Tornero, {\em Torsion of rational elliptic curves over quadratic fields.} Rev. R. Acad. Cienc. Exactas F\'is. Nat. Ser. A Math. RACSAM {\bf 108} (2014), 923--934.

\bibitem{GJT15}
E. Gonz\'alez--Jim\'enez and J.M. Tornero, {\em Torsion of rational elliptic curves over quadratic fields II.} Rev. R. Acad. Cienc. Exactas F\'is. Nat. Ser. A Math. RACSAM {\bf 110} (2016), 121--143. 

\bibitem{JKP04}
D. Jeon, C.H. Kim, and A. Schweizer, {\em On the torsion of elliptic curves over cubic number fields}, Acta
Arith. {\bf 113} (2004), 291--301.

\bibitem{JKP06}
D. Jeon, C.H. Kim, and E. Park, {\em  On the torsion of elliptic curves over quartic number fields}, J. London Math. Soc. (2) {\bf 74} (2006), 1--12.


\bibitem{K92}
S. Kamienny, {\em Torsion points on elliptic curves and $q$--coefficients of modular forms}. Invent. Math. {\bf 109} (1992) 129--133.

\bibitem{kenku39} M. A. Kenku, {\it The modular curve $X_0(39)$ and rational isogeny}, Math. Proc. Cambridge Philos. Soc. \textbf{85} (1979), 21--23.

\bibitem{kenku65} M. A. Kenku, {\it The modular curves $X_0(65)$ and $X_0(91)$ and rational isogeny}, Math. Proc. Cambridge Philos. Soc. \textbf{87} (1980), 15--20.

\bibitem{kenku169} M. A. Kenku, {\it The modular curve $X_0(169)$ and rational isogeny}, J. London Math. Soc. (2) \textbf{22} (1980), 239--244.

\bibitem{kenku125} M. A. Kenku, {\it The modular curve $X_0(125)$, $X_1(25)$ and $X_1(49)$}, J. London Math. Soc. (2) \textbf{23} (1981), 415--427.


\bibitem{KM88}
M.A. Kenku and F. Momose, {\em Torsion points on elliptic curves defined over quadratic fields}. Nagoya Math. J. {\bf 109} (1988) 125--149.

\bibitem{K97}
S. Kwon, {\em Torsion subgroups of elliptic curves over quadratic extensions}. J. Number Theory {\bf 62} (1997) 144--162.


\bibitem{Lozano}
A. Lozano--Robledo. {\em On the field of definition of p-torsion points on elliptic curves over the rationals}. Math. Ann. {\bf 357} (2013) 279--305.

\bibitem{Mazur1978}
B. Mazur, {\em Rational isogenies of prime degree}. Invent. Math. {\bf 44} (1978) 129--162.

\bibitem{Merel}
L. Merel, {\em Bornes pour la torsion des courbes elliptiques sur les corps de nombres}. Invent. Math. {\bf 124} (1996), 437--449.

\bibitem{MSZ89}
H.~M\"uller, H.~Str\"oher, and H.~Zimmer,
{\em Torsion groups of elliptic curves with integral j-invariant over quadratic fields}.
J. Reine Angew. Math. {\bf 397} (1989) 100--161.


\bibitem{twist}
F. Najman, {\em The number of twists with large torsion of an elliptic curve}. Rev. R. Acad. Cienc. Exactas Fís. Nat. Ser. A Math. RACSAM {\bf 109} (2015), 535--547. 

\bibitem{N15a}
F. Najman, {\em Torsion of elliptic curves over cubic fields and sporadic points on $X_1(n)$}. Math. Res. Letters, {\bf 23} (2016) 245--272.

\bibitem{Olson74}
L.~Olson, {\em Points of finite order on elliptic curves with complex multiplication}.
Manuscripta Math. {\bf 14} (1974) 195--205.


\bibitem{PWZ97}
A.~Peth\H{o}, T.~Weis, and H.~Zimmer,
{\em Torsion groups of elliptic curves with integral j-invariant over general cubic number fields}.
Int. J. Algebra Comput. {\bf 7} (1997) 353--413.


\bibitem{Silverman} 
J. H. Silverman,  
{\em The arithmetic of elliptic curves}, Second Edition, Graduate Texts in Mathematics, vol. 106, Springer, Dordrecht, 2009.

\bibitem{advanced}
J-H. Silverman, 
{\em Advanced topics in the arithmetic of elliptic curves}. Graduate Texts in Mathematics, vol. 151. Springer-Verlag, New York, 1994.


\bibitem{Sutherland2}
A.V. Sutherland, {\em Computing images of Galois representations attached to elliptic curves}. Forum Math. Sigma {\bf 4} (2016), e4, 79 pp.


\bibitem{zywina} D. Zywina, {\em On the possible images of the mod $\ell$ representations associated to elliptic curves over $\Q$}. \href{http://arxiv.org/abs/1508.07660}{arXiv:1508.07660}.

\end{thebibliography}
\end{document}